\documentclass[12pt,oneside]{amsart}

\usepackage{amssymb}
\usepackage{verbatim}
\usepackage[T1]{fontenc}
 \usepackage{pifont}
 \usepackage{color}
\usepackage{tikz}
\usetikzlibrary{decorations.pathreplacing}

\numberwithin{equation}{section}
\newtheorem{theorem}{Theorem}[section] 
\newtheorem{lemma}[theorem]{Lemma}
\newtheorem{proposition}[theorem]{Proposition}
\newtheorem{corollary}[theorem]{Corollary}

\theoremstyle{definition}

\newtheorem{remark}[theorem]{Remark}
\newtheorem{example}[theorem]{Example}

\newtheorem{notation}[theorem]{Notation}

\newcommand{\popo}{\mathbb{P}^1 \times \mathbb{P}^1}

\newcommand{\pr}{\mathbb{P}}
\newcommand{\supp}{\operatorname{Supp}}

\newcommand{\W}{\mathcal W}

\begin{document}


\title[The minimal free resolution of fat almost complete intersections in
$\popo$]{The minimal free resolution of fat almost complete
intersections in $\popo$}

\author{Giuseppe Favacchio}
\address{Dipartimento di Matematica e Informatica\\
	Viale A. Doria, 6 - 95100 - Catania, Italy} \email{favacchio@dmi.unict.it} \urladdr{}

\author{Elena Guardo}
\address{Dipartimento di Matematica e Informatica\\
Viale A. Doria, 6 - 95100 - Catania, Italy}
\email{guardo@dmi.unict.it}
\urladdr{http://www.dmi.unict.it/$\sim$guardo/}

\thanks{Version: Sept 15, 2016}

\begin{abstract}
A current research theme is to compare symbolic powers of an ideal
$I$ with the regular powers of $I$. In this paper, we focus on the
case that $I=I_X$ is an ideal defining an almost complete
intersection (ACI) set of points $X$ in $\popo$. In particular,
we describe a minimal free bigraded resolution of a non
arithmetically Cohen-Macaulay (also non homogeneus) set $\mathcal Z$  of fat
points whose support is an ACI generalizing Corollary 4.6 given in \cite{MFO} for homogeneous sets of triple points. We call $\mathcal Z$ a fat ACI. We also show that its symbolic and ordinary powers are equal, i.e, 
$I_{\mathcal Z}^{(m)}=I_{\mathcal Z}^{m}$ for any $m\geq 1.$

\end{abstract}

\maketitle

\renewcommand{\thetheorem}{\thesection.\arabic{theorem}}
\setcounter{theorem}{0}

\section{Introduction}
A research problem of interest regarding which symbolic powers of
ideals are contained in a given ordinary power of the ideal have
recently been studied in \cite{BCH,BH,BH2,HH}, with a focus on
ideals defining 0-dimensional subschemes of projective space.

Inspired by recent papers of \cite{MFO, GHVT1, GHVT2, GHVT3}, we focus on the case that $I$ is an ideal defining a
set of points in $\popo$ since, in particular, $I$ can be
considered as a set of particular lines in $\pr^3$.

Throughout this paper,  the polynomial ring
$R:=k[x_0,x_1,x_2,x_3]$ with the bigrading given by $\deg x_0 =
\deg x_1 = (1,0)$ and $\deg x_2 = \deg x_3 = (0,1)$ is the
coordinate ring of $\popo$. A point is denoted by $P = [a_0 : a_1]
\times [b_0 : b_1]$ in $\mathbb{P}^1 \times \mathbb{P}^1$ and it
is defined by the bihomogeneous ideal $I_P = (a_1 x_0 - a_0 x_1,
b_1 x_2 - b_0 x_3)$.   A set of points $X = \{P_1,\ldots,P_s\}
\subseteq \popo$ is then associated to the bihomogeneous ideal
$I_X = \bigcap_{P \in X} I_P$. If we only consider the standard
grading of this ideal, then $I_X$ defines a union $X$ of lines in
$\pr^3$. Given a set of distinct points $X=\{P_1,\dots,P_s\}$ and
positive integers $m_1,\dots,m_s$, we call $Z=m_1P_1+\dots+m_sP_s$ a
{\it set of fat points supported at $X$}.

Given a homogeneous ideal $I \subset R$, the $m$-th symbolic power
of $I$ is the ideal $I^{(m)} = R\cap (\cap_{P\in Ass(I)}(I^m R_P
))$. Following \cite{BH2}, an ideal of the form $I =
\cap_{i}(I_{P_i}^{m_i}$ ) where $P_1,\dots , P_n$ are distinct
points of $\popo$, $ I_{P_i}$  is the ideal generated by all forms
vanishing at $P_i$ and each $m_i$ is a non-negative integer,
$I^{(m)}$ turns out to be $\cap_{i}(I_{P_i}^{mm_i})$. If $I^m$ is
the usual power, then there is clearly a containment $I^m
\subseteq I^{(m)}$ and  a much more difficult problem is to
determine when there are containments of the form
$I^{(m)}\subseteq I^r$. Furthermore, the $m$-th symbolic power of
$I_X$ has the form $I_{X}^{(m)} =\cap_{i=1}^s I_{P_i}^m$. The
scheme defined by $I_{X}^{(m)}$ is sometimes referred to as a {\it
homogeneous} set of fat point and denoted by $mP_1 + \dots +
mP_s$.

We say that a set of points $X$ in $\popo$  is {\it arithmetically
Cohen-Macaulay} (ACM) if its coordinate ring $R/I_X$ is
Cohen-Macaulay. A set of points $X$ is a {\it complete intersection} if
$I_X$ is a complete intersection. We write that $X=CI(a,b)$ if
$I_X$ is generated by a form of degree $(a,0)$ and a form of
degree $(0,b).$ The set $X$ is an {\it almost complete intersection}
(ACI) if the number of minimal generators is one more than the
codimension of $X$, i.e., $X$ has three minimal generators.

 Let $X$ be an almost complete intersection in $\popo$ and let $Z=m_1P_1+\dots+m_sP_s$ be a 
 set of fat points supported at $X$. We call $\mathcal Z$ a {\it fat almost complete intersection}.

A classification of reduced and fat ACM sets of points of $\popo$
can be found in \cite{GVbook} Theorem 4.11 and Theorem 6.21,
respectively.

In this paper,  we focus on the study of special sets of fat points $\mathcal Z$ whose support is either  ACM or non ACM. In particular,  we give a minimal free bigraded resolution of $\mathcal Z$ in both cases (see Theorem \ref{Start1} and Theorem \ref{resolution}).

In \cite{GHVT2},  Theorem 1.1 the authors proved the following
\begin{theorem}[Theorem 1.1, \cite{GHVT2}]\label{cubo}
Let  $X\subseteq \popo$ be an ACM set of points.
Then $I_X^m =I_X^{(m)}$ for all $m \geq 1$
if and only if  $I_X^3 =I_X^{(3)}$.
\end{theorem}

In \cite{MFO}, the authors proposed a classification of the sets of points $X \subseteq \mathbb{P}^1
\times \mathbb{P}^1$ satisfying $I_X^3 = I_X^{(3)}.$ We
require the following notation. Let $\pi_1:\popo \rightarrow
\pr^1$ denote the natural projection
\[P = A \times B \mapsto A.\]
If $X \subseteq \popo$ is a finite set of reduced points, let
$\pi_1(X) = \{H_1,\ldots,H_r\}$ be the set of distinct first
coordinates that appear in $X$.  For $i=1,\ldots,h$, set $\overline{\alpha_i}
= |X \cap \pi_1^{-1}(H_i)|$, i.e., the number of points in $X$
whose first coordinate is $H_i$.  After relabeling the $H_i$'s so
that $\overline{\alpha_i} \geq \overline{\alpha_{i+1}}$ for $i = 1,\ldots,r-1$, we set
$\alpha_X = (\overline{\alpha_1},\ldots,\overline{\alpha_r})$. In particular,  they  proved the following two results:

\begin{corollary}[Corollary 4.4, \cite{MFO}] \label{P1xP1 result}
Let $X \subseteq \mathbb{P}^1 \times \mathbb{P}^1$ be any ACM set
of points.  Then
\medskip
\begin{itemize}
\item[(a)] $I_X^2 = I_X^{(2)}$.
\medskip
\item[(b)] The following are equivalent:
\medskip
\begin{itemize}
\item[(i)] $I_X^2$ defines an ACM scheme;
\item[(ii)] $I_X^3 = I_X^{(3)}$ is the saturated ideal of an ACM scheme;
\item[(iii)] $X$ is a complete intersection;
\item[(iv)] $\alpha_X = (a,a,\ldots,a)$ for some integer $a \geq 1$.
\end{itemize}
\medskip
\item[(c)] The following are equivalent:
\medskip
\begin{itemize}
\item[(i)] $I_X^3 = I_X^{(3)}$ is the saturated ideal of a non-ACM scheme;
\item[(ii)] $I_X$ is an almost complete intersection;
\item[(iii)] $\alpha_X = (a,\ldots,a,b,\ldots,b)$ for integers $a > b \geq 1$.
\end{itemize}
\end{itemize}
\end{corollary}

and

\begin{corollary}[Corollary 4.6, \cite{MFO}]\label{restripli}
Let $Z \subseteq \popo$ be a homogeneous set of triple
points (i.e., where every point has multiplicity three) and let  $X$ be the support of $Z$.  If $I_X$ is an almost
complete intersection with $\alpha_X =
(\underbrace{a,\ldots,a}_c,\underbrace{b,\ldots,b}_d)$,
then $I_Z$ has a bigraded minimal free resolution of the form
\[0 \rightarrow F_2 \rightarrow F_1 \rightarrow F_0 \rightarrow I_Z \rightarrow 0\]
where
\footnotesize
\begin{eqnarray*}
F_0 &  = & R(-3c-3d,0)\oplus R(-3c-2d,-b) \oplus R(-2c-2d,-a) \oplus R(-3c-d,-2b) \oplus R(-2c-d,-b-a) \oplus \\
&&R(-c-d,-2a)\oplus R(-3c,-3b) \oplus R(-2c,-2b-a) \oplus R(-c,-b-2a) \oplus R(0,-3a)\\
F_1 & = &R(-c,-3a)\oplus R(-2c,-2a-b) \oplus R(-3c,-a-2b) \oplus R(-c-d,-2a-b)   \oplus R(-2c-d,-a-2b) \oplus \\
&&R(-3c-d,-3b) \oplus R(-2c-d,-2a) \oplus R(-3c-d,-a-b)\oplus R(-2c-2d,-a-b) \oplus \\
&& R(-3c-2d,-2b) \oplus R(-3c-2d,-a) \oplus R (-3c-3d,-b)\\
F_2 & = & R(-3c-2d,-b-a) \oplus R(-3c-d,-a-2b) \oplus R(-2c-d,-2a-b).
\end{eqnarray*}
\normalsize
\end{corollary}

Here, we generalize Corollary \ref{restripli} for a special set $\mathcal Z$ of fat points whose support is an almost complete intersection (ACI), i.e. for a special fat almost complete intersection. We note that we don't require that $\mathcal Z$ is homogeneous. To shorten the notation we will say $\mathcal Z$ is a fat ACI.

Let $X$ be an ACI set of distinct points in $\popo$ such that $\alpha_X =
(\underbrace{a,\ldots,a}_{\alpha_1},\underbrace{b,\ldots,b}_{\alpha_2})$
for two integers $a > b \geq 1$. Set $a:=\beta_1+\beta_2,$ $b:=\beta_1$ and
$r=\alpha_1+\alpha_2$, so that $\alpha_X =
(\underbrace{\beta_1+\beta_2,\ldots,\beta_1+\beta_2}_{\alpha_1},\underbrace{\beta_1,\dots,\beta_1}_{\alpha_2})$.

Let $H_i$ be horizontal lines of type $(1,0)$ and $V_j$ vertical lines of type $(0,1)$, then a point in $\popo$ can be denoted by $P_{ij}:=H_i\times V_j.$ If $\pi_1(X) = \{H_1,\dots,H_r\}$ and $\pi_2(X) = \{V_1, \dots,
V_a\}$,  then $X\subset W = \{P_{ij}\ |\ i=1,\ldots,r \text{ and } j=1,\ldots,a \}$. Note that $W$ is a complete
intersection of reduced points.

Define $\mathcal{Z}:=\sum w_{ij}P_{ij}$ a fat ACI
of $\mathbb{P}^1\times\mathbb{P}^1$ where 
\begin{equation}\label{fatZ}
w_{ij}=\begin{cases}
m_{11} & \text{if}\ (i,j)\le(\alpha_1, \beta_1)\\
m_{21} & \text{if}\ (\alpha_1+1, 1)\le(i,j)\le(\alpha_1+\alpha_2, \beta_1)\\
m_{12} & \text{if}\ (1, \beta_1+1)\le(i,j)\le(\alpha_1, \beta_1+\beta_2)\\
0 & \text{otherwise}\end{cases}\end{equation} for some non negative integers $m_{11},
m_{12}, m_{21}.$  Renumbering the lines $H_i$ or $V_j$, we can
always assume that $m_{21}\le m_{12}.$

The following picture shows how $\mathcal{Z}$ looks like.
\begin{center}
\begin{tikzpicture}
   \draw[thick] (0,2) --  (0,4);    
   \draw[thick] (-4,0) --  (-4,4);    
  \draw[thick] (-4,0) --  (-2,0);    
    \draw[thick] (-2,0) --  (-2,4);    
     \draw[thick] (-4,2) --  (0,2);    
    \draw[thick] (-4,4) --  (0,4);   

   \draw[decorate,decoration={brace,amplitude=4pt}](-9/2,2)--(-9/2,4);
    \node [align=left] at (-5,3)  {$\alpha_1$};    
   \draw[decorate,decoration={brace,amplitude=4pt}](-9/2,0)--(-9/2,2);
   \node [align=left] at (-5,1) {$\alpha_2$};    
   \node [align=center] at (-3,3)  {$m_{11}$};
   \node [align=center] at (-3,1)  {$m_{21}$};
   \node [align=center] at (-1,3)  {$m_{12}$};

   \draw[decorate,decoration={brace,amplitude=4pt},xshift=0pt,yshift=0pt](-4,9/2)--(-2,9/2);
   \draw[decorate,decoration={brace,amplitude=4pt}](-2,9/2)--(0,9/2);
   \node [align=center] at (-3,5)  {$\beta_{1}$};
   \node [align=center] at (-1,5)  {$\beta_{2}$};

\node [align=left] at (-7,2)  {$\mathcal Z=$};
  \end{tikzpicture}
\end{center}

We denote by $\mathcal{Z}_1:=\sum \overline{w}_{ij}P_{ij}$ a set
of fat points of $\mathbb{P}^1\times\mathbb{P}^1$ where
$$\overline{w}_{ij}=\begin{cases}
(m_{11}-1)_+ & \text{if}\ (i,j)\le(\alpha_1, \beta_1)\\
(m_{21}-1)_+ & \text{if}\ (\alpha_1+1, 1)\le(i,j)\le(\alpha_1+\alpha_2, \beta_1)\\
m_{12} & \text{if}\ (1, \beta_1+1)\le(i,j)\le(\alpha_1, \beta_1+\beta_2)\\
0 & \text{otherwise}\end{cases}$$ for $m_{11}, m_{12}, m_{21}$ as
in $\mathcal{Z}$ and $(n)_+:=max\{n,0\}.$

The main result of this paper is:
\begin{theorem}[Theorem \ref{resolution}]
Let $0\to\mathcal{L}_2\to\mathcal{L}_1\to\mathcal{L}_0\to R\to
R/I_{\mathcal Z_1}\to 0$ be a minimal free resolution of
$I_{\mathcal{Z}_1}.$ Then a minimal free resolution of a fat ACI of type (\ref{fatZ})
$I_{\mathcal{Z_{}}}$ is

$$ 0\to \bigoplus_{(a,b-\beta_1)\in \mathcal{A}_1(\mathcal{Z})} R(-a,-b)\oplus \mathcal{L}_2(0,-\beta_1)\to$$
$$ \to \bigoplus_{(a,b-\beta_1)\in \mathcal{A}_0(\mathcal{Z})} R(-a,-b)\bigoplus_{(a,b)\in \mathcal{A}_1(\mathcal{Z})} R(-a,-b)\oplus \mathcal{L}_1(0,-\beta_1)\to$$
$$\to \bigoplus_{(a,b)\in \mathcal{A}_0(\mathcal{Z})} R(-a,-b) \oplus \mathcal{L}_0(0,-\beta_1)\to I_{\mathcal{Z}}\to 0$$

Where
$\mathcal{A}_0(\mathcal{Z})=\{(\alpha_1(m_{11}+i)+\alpha_2m_{21},
((m_{12}-m_{11})_+-i)\beta_2)\ |\  i=0,\ldots,  (m_{12}-m_{11})_+
\}$
and
$\mathcal{A}_1(\mathcal{Z})=
\{(\alpha_1(m_{11}+i+1)+\alpha_2m_{21},
((m_{12}-m_{11})_+-i)\beta_2)\ |\ i=0,\ldots, (m_{12}-m_{11})_+-1
\}.$
\end{theorem}

That is, if we set $\mu=min(m_{11},m_{21})$  recursively, we find a minimal bigraded free resolution
of non homogeneous sets of fat points $\mathcal Z_i\subset
\mathcal Z$ whose support is an almost complete intersection for all
$i=0,\dots,\mu$  but  $\mathcal Z_{\mu}$. In particular,
$\mathcal Z_0=\mathcal Z$ and the base case  $\mathcal Z_{\mu}$  can be of two types

\begin{center}
\begin{tikzpicture}
   \draw[thick] (0,2) --  (0,4);    
   \draw[thick] (-4,0) --  (-4,4);    
  \draw[thick] (-4,0) --  (-2,0);    
    \draw[thick] (-2,0) --  (-2,4);    
     \draw[thick] (-4,2) --  (0,2);    
    \draw[thick] (-4,4) --  (0,4);   

\node [align=left] at (-5,2)  {$1)$};

   \node [align=center] at (-3,3)  {$m_{11}-m_{21}$};
   \node [align=center] at (-3,1)  {$0$};
   \node [align=center] at (-1,3)  {$m_{12}$};



\node [align=left] at (3,2)  {or};
\node [align=left] at (4,2)  {$2)$};

 \draw[thick] (5,2) --  (9,2);    
   \draw[thick] (5,4) --  (9,4);    
  \draw[thick] (5,0) --  (7,0);    
    \draw[thick] (7,0) --  (7,4);    
     \draw[thick] (5,0) --  (5,4);    
    \draw[thick] (9,2) --  (9,4);   

   \node [align=center] at (6,1)  {$m_{21}-m_{11}$};
   \node [align=center] at (6,3)  {$0$};
   \node [align=center] at (8,3)  {$m_{12}$};

  \end{tikzpicture}
\end{center}
\begin{enumerate}
\item[1)] if $m_{11}> m_{21}$  then $\mathcal Z_{\mu}$ is an ACM set fat points
supported on a complete intersection $CI(\alpha_1,\beta).$ From \cite{GVbook}, Theorem 6.21 we can recover its minimal bigraded free resolution;
\item[2)] if $m_{11}< m_{21}$ then $\mathcal Z_{\mu}$ is
not ACM. In this case Lemma \ref{Start1} gives a minimal free bigraded resolution of $\mathcal Z_{\mu}$. In particular,  in this second case, the support $X$ of $\mathcal Z_{\mu}$ is the disjoint union of two complete intersections $X_1=CI(\alpha_1,\beta_2)$ and $X_2=CI(\alpha_2,\beta_1)$
\item[3)] The case $m_{11}=m_{21}$ is shown in Corollary \ref{corAlmostHom}. In this case, the support of $\mathcal Z_{\mu}$ is a $CI(\alpha_1,\beta_2)$.
\end{enumerate}

We also note that Theorem \ref{resolution} in the case $m_{11}=m_{12}=m_{21}=3$ gives Corollary \ref{restripli} proved in \cite{MFO}.

In Section \ref{powerZ}, Theorem \ref{compowerZZ}, we prove that $I_{\mathcal Z}^{(m)}=I_{\mathcal
Z}^m$ for any positive integer $m$ where $\mathcal Z$ is a fat ACI of type (\ref{fatZ}).  This result gives a new class
of non ACM set of fat points in $\popo$ whose symbolic and regular
powers are equals.

\noindent
{\bf Acknowledgments.} We gratefully acknowledge the computer algebra systems CoCoA \cite{C} and Macaulay \cite{Mt} that inspired many of the
results of this paper. We also thank the referee for his/her useful comments.

\section{Background and notation}

In this section, we recall some well-known facts about ACM sets of
fat points in $\mathbb{P}^1 \times \mathbb{P}^1.$ Then we start
the study of a set $\mathcal W$ of three non collinear fat points of $\popo.$  We observe that  $\supp(\W)$ of $\W$ is ACI but $\W$ can be either ACM or not ACM. Proposition \ref{nicegens} extends a property of the ACM set of points to our case of interest.

\begin{lemma}\label{resfatP}Let $P\in \mathbb{P}^1 \times \mathbb{P}^1$ be a point.
Then the bigraded minimal free resolution of $I(P)^m$ 
 is
$$0\to \bigoplus_{t=1}^{m}R(t-m-1,-t)\to \bigoplus_{t=0}^{m}R(t-m,-t) \to I(P)^m \to 0 $$
\end{lemma}
\begin{proof} This follows, for instance, from Theorem 6.27, \cite{GVbook}.
\end{proof}

From \cite{GVT}, Theorem 5.4 and Theorem 4.11, the following two results hold:

\begin{lemma}\label{Start2} In $\popo$ let $\mathcal{Z}$ be
    $$\mathcal{Z}:=\sum_{(1, 1)\le(i,j)\le(\alpha, \beta_1)} m_{11}P_{ij}+\sum_{(1, \beta_1+1)\le(i,j)\le(\alpha, \beta_1+\beta_2)} m_{12}P_{ij}$$
    a set of fat points whose support is $X=CI(\alpha, \beta)$
    where $\beta=\beta_1+\beta_2$.

    Set $M:=\max\{m_{11},{m_{12}}\}$, then a minimal free resolution of $I_{\mathcal{Z}}$ is
    {\Small $$0\to \bigoplus_{t=1}^{M} R(-\alpha t, -\beta_1(m_{11}-t+1)_+-\beta_2(m_{12}-t+1)_+)  \to  \bigoplus_{t=0}^{M} R(-\alpha t, -\beta_1(m_{11}-t)_+-\beta_2(m_{12}-t)_+) \to I_{\mathcal Z} \to 0$$}
\end{lemma}
\begin{proof}  $\mathcal{Z}$ is ACM and the tuple associated is

$$\alpha_{\mathcal{Z}}=(\underbrace{\gamma_0,\dots,\gamma_0\  }_{\alpha},\underbrace{\gamma_1,\ \ldots\gamma_1}_{\alpha},\ \ldots,\  \underbrace{\gamma_M, \ \ldots\gamma_M }_{\alpha} ).$$

\noindent where $\gamma_i:=(m_{11}-i)_+\beta_{11}+(m_{12}-i)_+\beta_{12}$.

\end{proof}

\begin{corollary} \label{resCI} With the notation as above, if
$m_{11}=m_{12}$, i.e., $\mathcal{Z} $ is a homogeneous set of fat
points whose support is $X=CI(\alpha, \beta)$, then a minimal free
resolution    is     $$0\to \bigoplus_{i=0}^{m-1}R(-(i+1)\alpha,-(m-i)\beta)\to \bigoplus_{i=0}^{m}R(-i\alpha ,-(m-i)\beta)\to I_{\mathcal{Z}} \to 0$$
\end{corollary}

To describe a minimal free bigraded resolution of a fat ACI $\mathcal Z$ of type (\ref{fatZ}),
we need to describe the minimal free bigraded resolution of a particular case of a fat ACI.  

We set our notation.
\begin{notation}\label{notationZ}
Let $\mathcal W$ be a fat ACI consisting only of three non collinear fat points
$P_{ij}:=H_i\times V_j$ with $H_i$ horizontal lines of type
$(1,0)$ and $V_j$ vertical lines of type $(0,1)$ for $i,j=1,2$.

We will assume $m_{21}\le m_{12}$ and  $(a)_+:=max\{a,0\}.$ Then $\mathcal
W:=m_{11}P_{11}+m_{21} P_{21}+m_{12} P_{12} $, and $\mathcal
W_1:=(m_{11}-1)_+P_{11}+ (m_{21}-1)_+ P_{21}+m_{12}P_{12}$ is the
set of points obtained from $\mathcal W$ by decreasing by 1 the
multiplicity of each point on  $V_1.$ 
\end{notation}

\begin{center}
\begin{picture}(150,110)(25,-10)
\put(0,60){$\mathcal W = $} \put(60,35){\line(0,1){45}}
\put(80,35){\line(0,1){45}}

\put(55,90){$V_1$} \put(75,90){$V_2$}

 \put(55,55){\line(1,0){45}}
 \put(55,75){\line(1,0){45}}

\put(35,50){$H_2$} \put(35,70){$H_1$}

\put(65,40){$$} \put(57.5,52){$\stackrel{\bf m_{21}}{\bullet\;
\; \; \; \; \; }$}

\put(65,60){$$} \put(57.5,72){$\stackrel{\bf m_{11}}{\bullet\;
\; \; \; \; \; }$} \put(63,78){$$}

\put(83,38){$$} \put(78,72){$\stackrel{\bf m_{12}}{\bullet\;
\; \; \; \; \; }$} \put(83,78){$$}




\end{picture}
\begin{picture}(150,110)(25,-10)
\put(0,60){$\mathcal W_1 = $} \put(60,35){\line(0,1){45}}
\put(80,35){\line(0,1){45}}

\put(55,90){$V_1$} \put(75,90){$V_2$}

\put(55,55){\line(1,0){45}} \put(55,75){\line(1,0){45}}

\put(35,50){$H_2$} \put(35,70){$H_1$}

\put(65,40){$$} \put(44.5,52){$\stackrel{_{\bf(m_{21}-1)_+}
}{\bullet\; }$}

\put(65,60){$$} \put(44.5,72){$\stackrel{_{\bf(m_{11}-1)_+
}}{\bullet\;  }$} \put(63,78){$$}

\put(83,38){$$} \put(77,72){$\stackrel{_{\bf
m_{12}}}{\bullet\; \; \; \; \; \; }$} \put(83,78){$$}




\end{picture}
\end{center}

If $m_{21}=0$ then $\W$ is an ACM set of collinear points and
everything is known (\cite{GVT}, Corollary 4.9 and Theorem 4.11).

In order to describe the homological invariants of $\W$ we start
by proving a proposition that holds for ACM finite sets of points
in $\mathbb{P}^1 \times \mathbb{P}^1,$  see for instance \cite{GVbook} Theorem
7.12.

\begin{proposition}\label{nicegens} With the notation as above, let $\W=m_{11}P_{11}+m_{21} P_{21}+m_{12} P_{12}$ be a set of three non collinear fat points in $\popo$, then $I_{\W}$ is minimally generated by a set of forms such that each of them is a product of powers of lines.
\end{proposition}
\begin{proof} We claim that $I_{\W}$ is generated by the set of bihomogeneous forms $$\mathcal{G}(\W)=\{H_1^{a_1}H_2^{a_2}V_1^{b_1}V_2^{b_2}\ |\ a_1+b_2\ge m_{12},\ a_2+b_1\ge m_{21},\ a_1+b_1\ge m_{11} \}.$$
It is easy to check that $H_1^{a_1}H_2^{a_2}V_1^{b_1}V_2^{b_2}\in \mathcal{G}(\W) $ iff $H_1^{a_1}H_2^{a_2}V_1^{b_1}V_2^{b_2}\in I_{\W}.$ 
On the other hand, we distinguish the following cases:
\begin{enumerate}
\item  If either $m_{12}=0$ or $m_{21}=0,$ then $\W$ is ACM and so
the statement is true.

\item Suppose $m_{12}>0$ and $m_{21}>0$ and let $F\in I_{\W}$ be a
bihomogeneous form of bidegree $(a,b).$  Since $F\in
(H_1,V_2)^{m_{12}}$ we get $F=\sum_i Q_i H_1^i V_2^{m_{12}-i}$
where either $Q_i=0$ or $\deg(Q_i)=(a-i,b-m_{12}+i).$ Moreover
$F\in (H_2,V_1)^{m_{21}}$ but $H_1^i V_2^{m_{12}-i}\notin  (H_2,V_1)^{m_{21}}$, and,
since $I_{\W}$ is bihomogeneous,  $Q_i$ have to belong to
$(H_2,V_1)^{m_{21}}$ for each $i,$ that means $Q_i=\sum_j T_{ij}
H_2^{m_{21}-j} V_1^j.$ Therefore   $$F=\sum_i\sum_j
T_{ij}H_2^{m_{21}-j} V_1^j H_1^i V_2^{m_{12}-i}=$$
$$=\underbrace{\sum_{i+j < m_{11} }T_{ij} H_1^iV_1^j H_2^{m_{21}-j}  V_2^{m_{12}-i}}_{F'}+\underbrace{\sum_{i+j \ge m_{11}} T_{ij} H_1^iV_1^j H_2^{m_{21}-j}  V_2^{m_{12}-i}}_{F^*}.$$
Note that $F^*\in (\mathcal{G}(\W)),$ so the claim follows if we also prove that
$F'\in (\mathcal{G}(\W)).$ Then
\begin{itemize}
	\item[i)] if $m_{11}=0$ we get $F'=0$ and we are done;
	\item[ii)] if $m_{11}>0,$ we proceed by induction on $s:=m_{12}+m_{21}.$ If $s\le m_{11} +1$ then $\W$ is ACM, by Theorem 4.8 in \cite{GVT}, and the statement is true. Suppose $s>m_{11}+1.$ 
	Denoted by $w_1=\min\{m_{12},m_{11}-1\},$ and by $w_2=\min\{m_{21},m_{11}-1\}$ then $$F'=H_2^{m_{21}-w_2}V_2^{m_{12}-w_1}\cdot\underbrace{\sum_{i+j < m_{11} }T_{ij} H_1^i V_1^j H_2^{w_2-j} V_2^{w_1-i}}_{F''}.$$ From $F'\in (H_1,V_1)^{m_{11}}$ we have $F''\in (H_1,V_1)^{m_{11}}.$ If $m_{12}>m_{11}-1,$ then $w_1+w_2<s$ and $F''\in I(\W'')$ where $\W''=m_{11}P_{11}+w_1P_{12}+ w_2P_{21}.$ By inductive hypothesis, the forms in  $\mathcal{G}(\W'')$ generate $I_{\W''}$ and, for some bihomogeneus polynomial $C_l,$ $F''=\sum C_l  H_1^{a_1}H_2^{a_2}V_1^{b_1}V_2^{b_2}.$ Then$$F'=\sum C_l H_1^{a_1}H_2^{a_2+m_{21}-w_2}V_1^{b_1}V_2^{b_2+m_{12}-w_1}$$ with the exponents satisfying the systems below

\[ \begin{cases}
    a_1+b_1\ge m_{11}&\\
    a_1+b_2\ge w_1&\\
    a_2+b_1\ge w_2&
\end{cases} \text{and then }\  \begin{cases}
    a_1+b_1\ge m_{11}&\\
    a_1+b_2+m_{12}-w_1 \ge m_{12}&\\
    a_2+b_1+m_{21}-w_2\ge m_{21}&\\
\end{cases}
\]

as we need.

In order to conclude the proof, we have to
 consider $m_{12}<m_{11}<s-1.$ In this case, note
that $F'\in I(\hat{\W}),$ where
$\hat{\W}=m_{11}P_{11}+m_{12}P_{12}+ m_{21}P_{21}+
(s-m_{11}-1)P_{22}$ that is an ACM set of points, by \cite{GVT} Theorem 4.8. So $F'\in (\mathcal{G}(I_{\W}))$.
\end{itemize}
\end{enumerate}
\end{proof}
\begin{notation} From now on we will denote by $\mathcal{G}(I_{\W})$ a minimal set of generators of $I_{\W}$ as in Proposition \ref{nicegens}.
\end{notation}


 The next results are immediate consequences of Proposition \ref{nicegens}. Since $I_{\W_1}$ is still in the hypothesis of Proposition \ref{nicegens}, it suffices to prove them just for the product of powers of $H_1,H_2,V_1$ and $V_2.$

\begin{proposition}\label{nicesplit}With the notation as above, $$I_{\W}=V_1I_{\W_1}+H_1^{m_{11}}H_2^{m_{21}}\cdot(H_1,V_2)^{(m_{21}-m_{11})_+}$$
\end{proposition}

\begin{proposition}\label{nicecap}With the notation as above,
$$V_1I_{\W_1}\cap H_1^{m_{11}}H_2^{m_{21}}\cdot(H_1,V_2)^{(m_{12}-m_{11})_+} =V_1H_1^{m_{11}}H_2^{m_{21}}\cdot(H_1,V_2)^{(m_{12}-m_{11})_+}$$
\end{proposition}

The following proposition will give us a way to construct a free
resolution of $I_{\W}.$

\begin{proposition}\label{MV} The following sequence is exact:
\Small $$0\to
V_1H_1^{m_{11}}H_2^{m_{21}}\cdot(H_1,V_2)^{(m_{12}-m_{11})_+} \to
V_1I_{\W_1}\oplus
H_1^{m_{11}}H_2^{m_{21}}\cdot(H_1,V_2)^{(m_{12}-m_{11})_+} \to
I_{\W} \to 0 $$
\end{proposition}
\begin{proof} This follows from the exact sequence
$$0\to I\cap J \to I\oplus J \to I+J\to 0$$
(where $I,J$ are R-modules), Proposition \ref{nicesplit} and Proposition \ref{nicecap}.
\end{proof}

\begin{remark}\label{indstep} As a consequence of Proposition \ref{MV} and the mapping cone
construction, if $0\to
L_2\to L_1\to L_0$ is a minimal free resolution of $I_{\W_1}$
then, it is easy to compute a free resolution for $I_{\W}$ is {\Small
$$ 0\to \bigoplus_{(a,b)\in A_2(\W)}R(-a,-b)\oplus L_2(0,-1) \to$$
$$\to  \bigoplus_{(a,b)\in A_1(\W)}R(-a,-b)^2
\oplus R(-m_{11}-m_{21},-(m_{12}-m_{11})_+-1  )
\oplus L_1(0,-1)\to $$
\begin{equation}\label{res}\to \bigoplus_{(a,b)\in A_0(\W)}R(-a,-b)\oplus L_0(0,-1) \to I_{\W}\to 0 \end{equation}
}

where

{\Small
\begin{tabular}{l}
$A_0(\W):=\{(a,b)\ |\ a+b= m_{11}+m_{21}+ (m_{12}-m_{11})_+\ \ \text{and}\ \ 0\le b\le (m_{12}-m_{11})_+  \}$\\
$A_1(\W):=\{(a,b)\ |\ a+b=1+ m_{11}+m_{21}+ (m_{12}-m_{11})_+\ \text{and}\ 1 \le b\le (m_{12}-m_{11})_+ \}$\\
$A_2(\W):=\{(a,b)\ |\ a+b=2+ m_{11}+m_{21}+ (m_{12}-m_{11})_+\ \text{and}\ 2 \le b\le (m_{12}-m_{11})_+ +1 \}$\\
\end{tabular}}

We will show in Theorem \ref{minres} that the resolution will be minimal.
\end{remark}

From Remark \ref{indstep} we can describe the bigraded Betti numbers of
$I_{\W}$ when $m_{11}=0,$ i.e. $\W$ is a non ACM set of two non collinear
fat points. We note that in this case the support of $\W$ is not an ACI.
\begin{lemma}\label{startcaso1}
Let $\W=m_{12}P_{12}+m_{21}P_{21}$ be a set of two non collinear
fat points, then the minimal free resolution of $I_{\W}$ is:

    $$0\to \bigoplus_{_{\begin{array}{c}_{a+b=m_{12}+ m_{21}+2}\\ _{a,b\ge 2}\end{array} }}  R(-a,-b)^{\beta_{2}(a,b)} \to\bigoplus_{_{\begin{array}{c}_{a+b=m_{12}+ m_{21}+1}\\ _{ a,b\ge 1}\end{array} }} R(-a,-b)^{\beta_{1}(a,b)}\to$$
    $$\to\bigoplus_{_{\begin{array}{c}_{a+b=m_{12}+ m_{21}}\\ _{a,b\ge 0}\end{array} }}  R(-a,-b)^{\beta_{0}(a,b)}\to I_{\W}\to 0 $$

\noindent where $\beta_0(a,b):=\min\{a,b, m_{21}\}+1,$ $\beta_1(a,b):=\min\{a,b-1, m_{21}\}+\min\{a-1,b, m_{21}\}+1,$ and
$\beta_2(a,b):=\min\{a,b, m_{21}\}$.

\end{lemma}

\begin{proof} If $m_{21}=0$ then $\W$ consists of only one fat point and the statement is true by Lemma \ref{resfatP}.
Let us suppose $m_{21}>0$ and the statement true for $\W_1.$
From Remark \ref{indstep} we get that no
cancellation is numerically allowed in the resolution arising from the mapping cone construction, then by inductive hypothesis

{\Small
\begin{tabular}{l}
$\beta_{0}(a,b) =\begin{cases} \min\{ m_{21}-1+ m_{12}-(b-1), b-1, m_{21}-1\}+2 & \text{if}\ a+b-1= m_{21}-1+ m_{12},\; b\le m_{12} \\
                               \min\{ m_{21}-1+ m_{12}-(b-1), b-1, m_{21}-1\}+1 & \text{if}\ a+b-1= m_{21}-1+ m_{12},\; b> m_{12} \\
                                                               0 &  \text{otherwise}
                                \end{cases}$

\end{tabular}\\

\begin{tabular}{l}
$ =\begin{cases} \min\{ b, m_{21}\}+1 & \text{if}\ a+b= m_{21}+ m_{12},\; b\le m_{12} \\
a+1 & \text{if}\ a+b= m_{21}+ m_{12},\; b> m_{12} \\
0 &  \text{otherwise}
\end{cases}$
\end{tabular}
\begin{tabular}{l}
    $ =\begin{cases} \min\{a, b, m_{21}\}+1 & \text{if}\ a+b= m_{21}+ m_{12} \\
    0 &  \text{otherwise}
    \end{cases}$
\end{tabular}

}
as required.

Analogously we can compute $\beta_{1}(a,b)$ and $\beta_{2}(a,b).$
\end{proof}


\begin{theorem}\label{minres}Let $0\to F_2\to F_1\to F_0$ be the free resolution of $I_{\W}$ as in Remark \ref{indstep}, then no cancellation is  allowed.
\end{theorem}
\begin{proof}Let be $0\to \bar{F}_2\to \bar{F}_1\to \bar{F}_0$ be a minimal free resolution of $I_{\W}$. Then we first observe that
    $\dim_k \left(\bar{F}_0\right)_{(a,b)}= \dim_k (F_0)_{(a,b)},$ i.e. $\mathcal{G}\left(I_{\W}\right)=
    V_1\cdot\mathcal{G}\left(I_{\W_1}\right)\cup H_1^{m_{11}}H_2^{m_{21}}\cdot\mathcal{G}\left((H_1,V_2)^{(m_{12}-m_{11})_+}\right)$
    and it is a minimal set of generators for $I_{\W}.$   From Proposition \ref{nicegens}, it is easy to check that $\mathcal{G}\left(I_{\W}\right)\subseteq V_1\cdot\mathcal{G}\left(I_{\W_1}\right)\cup H_1^{m_{11}}H_2^{m_{21}}\cdot\mathcal{G}\left((H_1,V_2)^{(m_{12}-m_{11})_+}\right).$ On the other hand 
take $W\in I_{{\W_1}}$ and $G\in
\left((H_1,V_2)^{(m_{12}-m_{11})_+}\right)$   such that $V_1W+
H_1^{m_{11}}H_2^{m_{21}}G=0$ then $G\in (V_1).$ Hence let
$G=\sum_{i,j} T_{ij}H_1^iV_2^j,$ for some $T_{ij}\neq 0,$ and let
$P:=(H_u\times V_1)\notin \W$ be such that $T_{ij}\notin (H_u).$
We set $H_1^iV_2^j(P)=\alpha_{ij}\neq 0$ so we get  $\sum
T_{ij}\alpha_{ij}\in (V_1)$ and, because the bihomogenity of
$I_{\W}$, this implies that all $T_{ij}\in (V_1).$ Then $G=V_1G'$
and $W=- H_1^{m_{11}}H_2^{m_{21}}G'.$ Thus,  if a cancellation is
allowed it has to involve $F_2$ and $F_1.$ If $m_{21}-m_{12}+1\ge
m_{21}$ then $\W$ is aCM and we are done. We will show that no
cancellation is numerically allowed also in the not aCM case. We
proceed by induction on $m_{11}.$ If $m_{11}=0$ then the statement
is true from Lemma \ref{startcaso1}. Now we suppose $m_{11}>0.$ If
for some $(a',b')$ we have $\dim_k (F_1)_{(a',b')}\neq 0$ and
$\dim_k (F_{2})_{(a',b')}\neq 0$ then two cases can be
distinguished
\begin{enumerate}
    \item $\dim_k (L_1)_{(a',b'-1)}\neq 0$ and $\dim_k (L_2)_{(a',b'-1)}= 0$
    \item $\dim_k (L_1)_{(a',b'-1)}=0$ and $\dim_k (L_2)_{(a',b'-1)}\neq 0$
\end{enumerate}

\noindent where $0\to L_2\to L_1\to L_0$ is a minimal free resolution of
$I_{\W_1}.$
By Remark \ref{indstep} and using the same notation, the first case happens  if $(a',b')\in A_2(\W)\neq \emptyset$  so it must be $m_{12}>m_{11}$ and $a'+b'=2+m_{21}+ m_{12}.$ 
If $m_{11}=1$ then we get a contradiction since in this case, by
Lemma \ref{startcaso1}, we get $\dim_k (L_1)_{(a',b'-1)} \neq 0$
if and only if $a'+b'-1=m_{12}+(m_{21}-1)+1.$  We can assume
$m_{11}>1$ and we set
$\W_2:=(m_{11}-2)_+P_{11}+(m_{21}-2)_+P_{21}+m_{12}P_{12}.$
From $\dim_k (L_2)_{(a',b'-1)}=0$ 
we have  $(a',b'-1)\notin  A_2(\W_1),$ but $(a',b')\in A_2(\W),$ and then the only case we need to consider is 
 $(a',b')=(m_{12}+m_{21}, 2).$  Since $(a',b'-1)\notin A_1(\W_1)$ we have $\dim_k(L_1)_{(a',1)}\neq 0$ 
and again since $(a',0)\notin A_1(\W_2).$ 
In the second case we can proceed in a similar way. First note that $(a',b')\in A_1\cup\{(m_{11}+m_{21}, (m_{12}-m_{11})_+)\}$ i.e.
$\left\{\begin{array}{l} a'+b' =1+m_{11}+m_{21}+(m_{12}-m_{11})_+\\
1\le b' \le  (m_{12}-m_{11})_++1\end{array}\right. .$

Moreover, since $\dim_k(L_1)_{(a',b'-1)}=0$ 
 then $(a',b'-1)\notin A_1(\W_1)$ i.e. either $a'+b'\neq m_{11}+m_{21}+(m_{12}-m_{11}+1)_+$ or $b'\notin\{2,\ldots,  (m_{12}-m_{11}-1)_++2\}. $ Since the second condition always holds we get $m_{12}< m_{11},$ and then $(a',b')=(m_{11}+m_{12},1).$ Then $\dim_k(L_2)_{(a',0)}\neq 0$ 
  that is not allowed for a finite set of points.

\end{proof}

The next example shows how to compute inductively a minimal bigraded resolution of $I_{\W}.$
\begin{example}\label{trepunti}
    Let be $\W=2P_{11}+4P_{12}+3P_{21},$ we set $\W_k:=(2-k) P_{11}+4P_{12}+(3-k)P_{21},$ for $k=1,2.$ We use Lemma \ref{startcaso1} to compute the resolution of $I_{\W_2}$ where $\W_2=4P_{12}+P_{21}$ is a set of two non collinear fat points.
{\Small
    $$0\to R(-5, -2)\oplus R(-4, -3)\oplus R(-3, -4)\oplus R(-2, -5)\to$$
    $$\to R(-5, -1)^2 \oplus R(-4, -2)^3\oplus R(-3, -3)^3\oplus R(-2, -4)^3\oplus\to R(-1, -5)^2\to$$
    \begin{equation}\label{ex1}\to R(-5, 0)\oplus R(-4, -1)^2\oplus R(-3, -2)^2\oplus R(-2, -3)^2\oplus R(-1, -4)^2\oplus R(0, -5)\to I_{\W_2}\to 0\end{equation}
} The next step is to compute a minimal free resolution for $I_{\W_1}$ where  $\W_1=P_{11}+4P_{12}+2P_{21}$. 
First, we shift all the degrees of the modules in the resolution \ref{ex1} by $(0,-1)$, then we compute  all the pairs $(i,j)$ in
$\mathcal{A}_0(\W_1)$ and add $R(-i,-j)$ among the generators' module; we compute all the pairs $(i,j)$ in $\mathcal{A}_1(\W_1)$ and add $R(-i,-j)$ among the first syzygies' module and, as last step, we compute  all the pairs $(i,j)$ in $\mathcal{A}_2(\W_1)$ and add $R(-i,-j)$ among the second syzygies's module of $\W_2.$
Thus, a minimal free resolution for $I_{\W_1}$ is
{\Small
    $$0\to R(-6, -2)\oplus R(-5, -3)^{2}\oplus R(-4, -4)^{2}\oplus R(-3, -5)\oplus R(-2, -6)\to$$
    $$\to R(-6, -1)^2\oplus R(-5, -2)^{4} \oplus R(-4, -3)^{5}\oplus R(-3, -4)^{4}\oplus R(-2, -5)^3\oplus\to R(-1, -6)^2\to$$
    \begin{equation}\label{ex2}\to R(-6,0)\oplus R(-5, -1)^{2}\oplus R(-4, -2)^{3}\oplus R(-3, -3)^{3}\oplus R(-2, -4)^2\oplus R(-1, -5)^2\oplus R(0, -6)\to I_{\W_1}\to 0\end{equation}
}
Finally, repeating the same procedure as above, i.e.,  shifting all the modules' degrees in the resolution (\ref{ex2}) by $(0,-1)$ and adding $R(-i,-j)$ with $(i,j)$ all the pairs in $\mathcal{A}_0(\W),$ $\mathcal{A}_1(\W),$ $\mathcal{A}_2(\W)$ among  the generators' module, first syzygies' module and second syzygies's module of $\W_1$, respectively, we get  a minimal free resolution of $I_{\W}$
{\Small
$$0\to R(-7,-2)\oplus R(-6, -3)^{2}\oplus R(-5, -4)^{2}\oplus R(-4, -5)^{2}\oplus R(-3, -6)\oplus R(-2, -7)\to$$
$$\to R(-7,-1)^2\oplus  R(-6, -2)^{4}\oplus R(-5, -3)^{5} \oplus R(-4, -4)^{5}\oplus R(-3, -5)^{4}\oplus R(-2, -6)^3\oplus\to R(-1, -7)^2\to$$
$$\to R(-7,0)\oplus R(-6,-1)^{2}\oplus R(-5, -2)^{3}\oplus R(-4, -3)^{3}\oplus R(-3, -4)^{3}\oplus R(-2, -5)^2\oplus R(-1, -6)^2\oplus R(0, -7)\to$$
$$\to I_{\W}\to 0$$
}
\end{example}

\section{The minimal free resolution of fat almost complete intersection in  $\popo$}

As said in the introduction, in this section we prove the main result of the paper that generalizes Theorem \ref{minres} for any fat almost complete intersection $\mathcal Z$. Recall our notation
\begin{notation}\label{NotZZ} Let $\alpha_1,\alpha_2,\beta_1,\beta_2$ be positive integers,
we denote by $\mathcal{Z}:=\sum w_{ij}P_{ij}$ a fat ACI of $\mathbb{P}^1\times\mathbb{P}^1$ where
$$w_{ij}=\begin{cases}
m_{11} & \text{if}\ (i,j)\le(\alpha_1, \beta_1)\\
m_{21} & \text{if}\ (\alpha_1+1, 1)\le(i,j)\le(\alpha_1+\alpha_2, \beta_1)\\
m_{12} & \text{if}\ (1, \beta_1+1)\le(i,j)\le(\alpha_1, \beta_1+\beta_2)\\
0 & \text{otherwise}\end{cases}$$ for some non negative integers $m_{11}, m_{12},
m_{21}$ and we  denote by $\mathcal{Z}_1:=\sum \overline{w}_{ij}P_{ij}$
a set of fat points of $\mathbb{P}^1\times\mathbb{P}^1$ where
$$\overline{w}_{ij}=\begin{cases}
(m_{11}-1)_+ & \text{if}\ (i,j)\le(\alpha_1, \beta_1)\\
(m_{21}-1)_+ & \text{if}\ (\alpha_1+1, 1)\le(i,j)\le(\alpha_1+\alpha_2, \beta_1)\\
m_{12} & \text{if}\ (1, \beta_1+1)\le(i,j)\le(\alpha_1, \beta_1+\beta_2)\\
0 & \text{otherwise}\end{cases}$$
for $m_{11}, m_{12}, m_{21}$ as in $\mathcal{Z}.$

\noindent
We  set $Q_1:=H_1\cdots H_{\alpha_1},$ $Q_2:=H_{\alpha_1+1}\cdots
H_{\alpha_1+\alpha_2}$ and $U_1:=V_1\cdots V_{\beta_1},$
$U_2:=V_{\beta_1+1}\cdots V_{\beta_1+\beta_2}.$
\end{notation}

We have the following lemma:

\begin{lemma}\label{ZZasZ}$I_{\mathcal{Z}}=(Q_1,U_1)^{m_{11}}\cap(Q_1,U_2)^{m_{12}}\cap(Q_2,U_1)^{m_{21}}.$
\end{lemma}
\begin{proof} $I_{\mathcal{Z}}$ is the intersection of three powers of homogeneous complete intersection ideals  and $I^m= I^{(m)}$ where $I$ is the ideal defining a complete intersection from  \cite{ZS}, Appendix 6, Lemma 5. We have
{\Small    $$I_{\mathcal{Z}}=\bigcap_{(i,j)\le(\alpha_1, \beta_1)}(H_i,V_j)^{m_{11}}\cap \bigcap_{(\alpha_1+1, 1)\le(i,j)\le(\alpha_1+\alpha_2, \beta_1)}(H_i,V_j)^{m_{21}}\cap \bigcap_{(1, \beta_1+1)\le(i,j)\le(\alpha_1, \beta_1+\beta_2)}(H_i,V_j)^{m_{12}}=$$
}{\footnotesize  $$=\left(\bigcap_{(i,j)\le(\alpha_1, \beta_1)}(H_i,V_j)\right)^{m_{11}}\cap\left( \bigcap_{(\alpha_1+1, 1)\le(i,j)\le(\alpha_1+\alpha_2, \beta_1)}(H_i,V_j)\right)^{m_{21}}\cap \left( \bigcap_{(1, \beta_1+1)\le(i,j)\le(\alpha_1, \beta_1+\beta_2)}(H_i,V_j)\right)^{m_{12}}.$$
}    \end{proof}

\begin{remark}\label{seeZ}
 All the results given in Section 2 can be generalized by replacing $H_i$  by $Q_i$ and $V_j$ by $U_j.$ 
\end{remark}

The following Lemma generalizes Lemma \ref{startcaso1}. That is, we compute a minimal free resolution of $\mathcal Z$ whose support is the disjoint union of two fat complete intersections and it is never ACM. As pointed out in the introduction, this is one of the starting base case to describe a minimal free resolution of $I_{\mathcal{Z}}$ by induction when $m_{11}<m_{21}$.

\begin{lemma}\label{Start1} In $\mathbb{P}^1\times\mathbb{P}^1,$ let
     $$\mathcal{Z}:=\sum_{(1, \beta_1+1)\le(i,j)\le(\alpha_1, \beta_1+\beta_2)} m_{12}P_{ij}+\sum_{(\alpha_1+1, 1)\le(i,j)\le(\alpha_1+\alpha_2, \beta_1)} m_{21}P_{ij}$$ be a set of fat points whose support is the disjoint union of two fat complete intersections. Then a minimal free resolution of $I_{\mathcal{Z}}$ is
    $$0\to \bigoplus_{(a,b,c,d)\in \mathcal{D}_2} R(-a\alpha_1-b\alpha_2, -c\beta_1-d\beta_2)\to \bigoplus_{(a,b,c,d)\in \mathcal{D}_1} R(-a\alpha_1-b\alpha_2, -c\beta_1-d\beta_2) \to $$
    $$ \to \bigoplus_{(a,b,c,d)\in \mathcal{D}_0} R(-a\alpha_1-b\alpha_2, -c\beta_1-d\beta_2) \to I_{\mathcal Z} \to 0$$
    where:\\
    $\mathcal{D}_0:=\{(a,b,c,d) |\ 0\le a,d \le m_{12},\ 0\le b,c \le m_{21},\ a+d=m_{12},\ b+c=m_{21}  \},$\\
    $\mathcal{D}_1:=\{(a,b,c,d) |\ 0\le a,d \le m_{12},\ 0\le b,c \le m_{21},\ (a+d=m_{12}+1,\ b+c=m_{21})\vee( a+d=m_{12},\ b+c=m_{21}+1 )\},$ and 
    $\mathcal{D}_2:= \{(a,b,c,d) |\ 0\le a,d \le m_{12},\ 0\le b,c \le m_{21},\ a+d=m_{12}+1,\ b+c=m_{21}+1  \}.$
\end{lemma}
\begin{proof}
This follows by induction on $m_{21},$ using Lemma \ref{resCI} and the mapping cone construction.
\end{proof}

\begin{theorem} \label{resolution}
With Notation \ref{NotZZ}, 
let $0\to\mathcal{L}_2\to\mathcal{L}_1\to\mathcal{L}_0$ be a minimal free resolution of  $I_{\mathcal{Z}_1}.$
Then a minimal free resolution of  a fat ACI $I_{\mathcal{Z}}$ is

$$ 0\to \bigoplus_{(a,b-\beta_1)\in \mathcal{A}_1(\mathcal{Z})} R(-a,-b)\oplus \mathcal{L}_2(0,-\beta_1)\to$$
$$ \to \bigoplus_{(a,b-\beta_1)\in \mathcal{A}_0(\mathcal{Z})} R(-a,-b)\bigoplus_{(a,b)\in \mathcal{A}_1(\mathcal{Z})} R(-a,-b)\oplus \mathcal{L}_1(0,-\beta_1)\to$$
$$\to \bigoplus_{(a,b)\in \mathcal{A}_0(\mathcal{Z})} R(-a,-b) \oplus \mathcal{L}_0(0,-\beta_1)\to I_{\mathcal{Z}}\to 0$$

Where

$\mathcal{A}_0(\mathcal{Z})=\{(\alpha_1(m_{11}+i)+\alpha_2m_{21},
((m_{12}-m_{11})_+-i)\beta_2)\ |\  i=0,\ldots,  (m_{12}-m_{11})_+
\}$

$\mathcal{A}_1(\mathcal{Z})=
\{(\alpha_1(m_{11}+i+1)+\alpha_2m_{21}),
((m_{12}-m_{11})_+-i)\beta_2)\ |\ i=0,\ldots, (m_{12}-m_{11})_+-1
\}$

\end{theorem}

\begin{proof} The proof uses Lemma \ref{resCI}, Remark \ref{seeZ} and Remark \ref{indstep}.
Note that, by induction, Lemma \ref{Start1} and Lemma \ref{Start2}, the number of elements in a minimal set
of generators for the modules in the resolution does not depend on $\alpha_1, \alpha_2, \beta_1, \beta_2.$
Moreover,  using  Remark \ref{indstep}, if $\alpha_1=\alpha_2=\beta_1=\beta_2=1$ we get 
$|\mathcal{A}_0(\mathcal{Z})| = |A_0(\W)|,$ $|\mathcal{A}_0(\mathcal{Z})|+|\mathcal{A}_1(\mathcal{Z})| = |A_1(\W)|$ and
$|\mathcal{A}_1(\mathcal{Z})| = |A_2(\W)|.$
Therefore by induction and Theorem \ref{minres}  no cancellation is allowed in the resolution arising from mapping cone. This follows since the maps of the mapping cone cannot have invertible entries otherwise by Remark \ref{seeZ}, the maps of the mapping cone used in Theorem \ref{minres} would also have invertible entries.

\end{proof}

\begin{example} Consider the following set of fat points with $\alpha_1=\beta_1=\beta_2=2$ and $\alpha_2=1.$
	$$\mathcal{Z}:=\begin{array}{cccc}
    \phantom{+}2P_{11}+ & 2P_{12}+& 4P_{13}+& 4P_{14}+\\
    +2P_{21}+ & 2P_{22}+& 4P_{23}+& 4P_{24}+\\
    +3P_{31}+ & 3P_{32}\phantom{+} & & \\
    \end{array}$$
    Note that $\mathcal{A}_0(\mathcal{Z})=\{(7,4), (9,2), (11,0) \}$ and 
    $\mathcal{A}_1(\mathcal{Z})= \{(9,4), (11,2)\}.$
        
    Set, for $i=0,1,2$
    $$\mathcal{Z}_i:=\begin{array}{cccc}
    \phantom{+}(2-i)P_{11}+ & (2-i)P_{12}+& 4P_{13}+& 4P_{14}+\\
    +(2-i)P_{21}+ & (2-i)P_{22}+& 4P_{23}+& 4P_{24}+\\
    +(3-i)P_{31}+ & (3-i)P_{32}\phantom{+}& & \\
    \end{array}$$

    We start by computing the resolution of $\mathcal{Z}_2$. By Lemma \ref{Start1} we get
    the following degrees for a minimal set of generators, first and second syzygies
    
{\Small    
$\begin{array}{ll}    
\text{Generators}:& \{(9,0),(8,2),(7,2), (6,4), (5,4),(4,6),(3,6), (2,8),(1,8),(0,10)\}\\
\text{First Syzygies}:& \{(9,2)^2,(8,4),(7,4)^2, (6,6), (5,6)^2,(4,8),(3,8)^2, (2,10),(1,10)\}\\
\text{Second Syzygies}:& \{(9,4),(7,6),(5,8),(3,10)\}
\end{array}$}

\noindent    where $(a,b)^n$ indicates that the set contains $n$ elements of degree $(a,b).$ 
    Now, by Theorem \ref{resolution}, and mimicking the procedure used in Example \ref{trepunti}, we can compute the resolution of $I_{\mathcal{Z}_1}$ where the degrees of a minimal   set of generators, first and second syzygies are respectively:
    
{\Small $\begin{array}{ll}
    \text{Generators}:& \{(9,2),(8,4),(7,4), (6,6), (5,6), (4,8), (3,8), (2,10), (1,10), (0,12)\}\cup\\ &\ \cup \{(10,0), (8,2), (6,4), (4,6)\}\\
    \text{First Syzygies}:& \{(9,4)^2,(8,6),(7,6)^2, (6,8), (5,8)^2,(4,10),(3,10)^2, (2,12),(1,12)\}\cup\\ &\ \cup \{(10,2)^2, (8,4)^2, (6,6)^2, (4,8)\}\\
   \text{Second Syzygies}:& \{(9,6),(7,8),(5,10),(3,12)\}\cup\{ (10,4), (8,6), (6,8)\}
\end{array}$}

 \noindent   Finally, applying again  Theorem \ref{resolution}, we get a minimal resolution of $I_{\mathcal{Z}}=I_{\mathcal{Z}_0}:$
    {\Small$$0\to [R(-10,-6)\oplus R(-9,-8)\oplus R(-8,-8)\oplus R(-7,-10)\oplus R(-6,-10)\oplus R(-5,-12)\oplus R(-3,-14)]\bigoplus$$
    $$\bigoplus [R(-9, -6)\oplus R(-11, -4)] \to$$
    $$ \to[\oplus R(-10,-4)^2\oplus R(-9,-6)^2\oplus R(-8,-8)\oplus R(-8,-6)^2\oplus R(-7,-8)^2\oplus R(-6,-10)\oplus R(-6,-8)^2\oplus $$
    $$\oplus R(-5,-10)^2\oplus R(-4,-10)\oplus R(-4,-12)\oplus R(-3,-12)^2\oplus R(-2,-14)\oplus R(-1,-14)]\bigoplus$$
    $$\bigoplus   [\oplus R(-7, -6)\oplus R(-9, -4)\oplus R(-11, -2)] \bigoplus   [ R(-9, -4)\oplus R(-11, -2)]  \to$$
    $$\to[R(-10,-2)\oplus R(-9,-4)\oplus R(-8,-6)\oplus R(-8,-4)\oplus R(-7,-6)\oplus R(-6,-8)\oplus R(-6,-6)\oplus R(-5,-8)\oplus$$
    $$\oplus R(-4,-10)\oplus R(-4,-8)\oplus R(-3,-10)\oplus R(-2,-12)\oplus R(-1,-12)\oplus R(0,-14)]\bigoplus$$
    $$\bigoplus[ R(-7, -4)\oplus R(-9, -2)\oplus R(-11, 0)]   \to I_{\mathcal{Z}} \to 0$$}

\end{example}

The next corollary better describes the resolution of
$I_{\mathcal{Z}}$ when $m_{11}=m_{21}.$
\begin{corollary}\label{corAlmostHom}
With Notation  \ref{NotZZ}, suppose $m_{11}=m_{21}=n$ and $m_{12}=m.$ 

    Then a minimal free resolution of $I_{\mathcal{Z}}$ is
   {\Small  $$0\to\bigoplus_{(a,b,c,d)\in \mathcal{B}_2(\mathcal{Z})} R(-a\alpha_1-b\alpha_2,-c\beta_1-d\beta_2)\to \bigoplus_{(a,b,c,d)\in \mathcal{B}_1(\mathcal{Z})} R(-a\alpha_1-b\alpha_2,-c\beta_1-d\beta_2)\to  $$
    $$\to \bigoplus_{(a,b,c,d)\in \mathcal{B}_0(\mathcal{Z})} R(-a\alpha_1-b\alpha_2,-c\beta_1-d\beta_2)\to I_{\mathcal{Z}}\to 0$$

        $\mathcal{B}_0(\mathcal{Z}):=\{(a,b,c,d)| a+d=m,\ b+c=n,\ 0\le b\le \min\{a,n\}\le a \le m \},$

        $\mathcal{B}_1(\mathcal{Z}):=\{(a,b,c,d)| (a+d=m+1,\ b+c=n)\vee (a+d=m,\ b+c=n+1),\ 0\le b\le \min\{a,n\}\le a \le m \},$

        $\mathcal{B}_2(\mathcal{Z}):=\{(a,b,c,d)| a+d=m+1,\ b+c=n+1,\ 0\le b\le \min\{a,n\} \le a \le m \}.$ }
\end{corollary}
\begin{proof} We proceed by induction on $n.$ If $n=0$ then  $\mathcal{Z}$ is homogeneous and its support is a complete intersection so, by Lemma \ref{resCI}, we are done.
Assume now $n>0$ and take $\mathcal{Z}_1$ as in Notation \ref{NotZZ}. Then we get
{\Small$$\cdots\to \bigoplus_{(a,b,c+1,d)\in \mathcal{B}_0(\mathcal{Z}_1) }R(-a\alpha_1-b\alpha_2,-c\beta_1-d\beta_2) \bigoplus_{(u,v)\in \mathcal{A}_0(\mathcal{Z}) }R(-u,-v)\to I_{\mathcal{Z}}\to 0$$

$\mathcal{B}_0(\mathcal{Z}_1):=\{(a,b,c,d)| a+d=m,\ b+c=n-1,\ 0\le b\le \min\{a,n-1\}\le a \le m \}$
{\normalsize  and}

$\mathcal{A}_0(\mathcal{Z})=\{((\alpha_1(n+i)+\alpha_2n), \beta_2(m-n-i))\ |\  i=0,\ldots,  m-n \}$}
i.e.
{\Small$$\cdots\to \bigoplus_{(a,b,c+1,d)\in \mathcal{B}_0(\mathcal{Z}_1 )}R(-a\alpha_1-b\alpha_2,-c\beta_1-d\beta_2) \bigoplus_{(a,b,c,d)\in \mathcal{A}'_0(\mathcal{Z}) }R(-\alpha_1a-\alpha_2b, -\beta_1c-\beta_2d)\to I_{\mathcal{Z}}\to 0$$}
where
$\mathcal{A}'_0(\mathcal{Z}):=\{(a,b,c,d)\ |\ b=n,\ c=0,\ a+d=m, \ n \le a \le m \}.$

Then
$\mathcal{B}_0(\mathcal{Z})=\mathcal{B}_0(\mathcal{Z}_1)\cup \mathcal{A}'_0(\mathcal{Z}).$ Analogously we get $\mathcal{B}_1(\mathcal{Z})$ and $\mathcal{B}_2(\mathcal{Z}).$

\end{proof}

Consequently, if $\mathcal{Z}$ is a homogeneous set of fat points, then a minimal free resolution is easy to describe. 
\begin{corollary}\label{ResACI}
With Notation  \ref{NotZZ}, suppose $m_{11}=m_{12}=m_{21}=m,$ i.e.   the support of $\mathcal{Z}$ is an almost complete intersection with associated tuple $$\alpha_{\mathcal{Z}}:=(\underbrace{\beta_1+\beta_2, \ldots,\beta_1+\beta_2}_{\alpha_1+\alpha_2},\underbrace{\beta_1,\ldots,\beta_1}_{\alpha_1}  )$$  for some $m\in \mathbb{N}.$

Then a minimal free resolution of $I_{\mathcal{Z}}$ is
$$0\to\bigoplus_{(a,b,c,d)\in \mathcal{B}_2(\mathcal{Z})} R(-a\alpha_1-b\alpha_2,-c\beta_1-d\beta_2)\to  \bigoplus_{(a,b,c,d)\in \mathcal{B}_1(\mathcal{Z})} R(-a\alpha_1-b\alpha_2,-c\beta_1-d\beta_2)\to  $$
$$\to \bigoplus_{(a,b,c,d)\in \mathcal{B}_0(\mathcal{Z})} R(-a\alpha_1-b\alpha_2,-c\beta_1-d\beta_2)\to I_{\mathcal{Z}}\to 0$$

$\mathcal{B}_0(\mathcal{Z}):=\{(a,b,c,d)| a+d=m,\ b+c=m,\ 0\le b\le a \le m\}$

$\mathcal{B}_1(\mathcal{Z}):=\{(a,b,c,d)| (a+d=m+1,\ b+c=m)\vee (a+d=m,\ b+c=m+1),\ 0\le b\le a \le m \}$

$\mathcal{B}_2(\mathcal{Z}):=\{(a,b,c,d)| a+d=m+1,\ b+c=m+1,\ 0\le b\le a \le m \}$
\end{corollary}
\begin{proof} Just use Corollary \ref{corAlmostHom}.    \end{proof}

\begin{remark}
Recently,  using Theorem \ref{cubo} and Corollary \ref{P1xP1 result}, it was proved in \cite{MFO},
that if $\mathcal{Z}$ is a homogeneous set of fat points whose
support is an almost complete intersection then,
$$I_{\mathcal{Z}}=
(Q_1,U_1)^{m}\cap(Q_1,U_2)^{m}\cap(Q_2,U_1)^{m}= J^{m}$$ \noindent
where we set $J:=(Q_1,U_1)\cap(Q_1,U_2)\cap(Q_2,U_1).$ That is,
the symbolic powers of $J$ and the regular powers are the same.
Therefore a proof of Corollary \ref{ResACI}  could  be given by
induction on $m$ since $J^{m}=J\cdot J^{m-1}.$
\end{remark}
In the next section we look at the symbolic powers of
$I_{\mathcal{Z}}$ in the non homogeneous case.

\section{Symbolic Vs Regular powers of a particular almost complete intersection}\label{powerZ}

As said in the introduction, given a homogeneous ideal $I$, the
$m$-th symbolic power of $I$ is the ideal $I^{(m)} = R\cap
(\cap_{P\in Ass(I)}(I^m R_P ))$. Following \cite{BH2}, for an
ideal of the form $I_X = \cap_{P_{ij}\in X}(I_{P_{ij}}^{m_{ij}}$ ) where $X\subseteq \popo$ is a finite set of points,  $ I_{P_{ij}}$  is the ideal generated by all forms vanishing at $P_{ij}$ and each $m_{ij}$ is a
non-negative integer, $I_X^{(m)}$ turns out to be
$\cap_{P_{ij}\in X}(I_{P_{ij}}^{mm_{ij}})$. If $I_X^m$ is the usual power, then we
have the containment $I_X^m \subseteq I_X^{(m)}$ and  it is a
difficult problem to determine when there are containments of
the form $I_X^{(m)}\subseteq I_X^r$. Furthermore, the $m$-th symbolic
power of $I_X$ has the form $I_{X}^{(m)} =\cap_{ij} I_{P_{ij}}^m$.

In this section we prove that if $\mathcal Z$ is a  fat ACI of type
(\ref{fatZ}), then $I_{\mathcal Z}^{(m)}=I_{\mathcal Z}^{m}.$
We start with the three non collinear points case by comparing the ideal $I_{\W}^m$ with
$I_{\W}^{(m)},$ where we denote by $I_{\W}^{(m)}:=I(m\cdot
m_{11}P_{11}+m\cdot m_{12}P_{12}+m\cdot m_{21}P_{21}).$

\begin{theorem}\label{compower}
Let $\W=m_{11}P_{11}+m_{12}P_{12}+m_{21}P_{21}$ be a fat ACI of three
non collinear  points in $\popo.$ Then $I_{\W}^{m}= I_{\W}^{(m)}.$
\end{theorem}
\begin{proof} 
	First note that $I_{\W}^{(m)}$, by Proposition \ref{nicegens} is generated by a set, $\mathcal{G}(I_{\W}^{(m)}),$ of forms which are product of lines. Thus, take such a form  $F:=H_1^{a_1}H_2^{a_2}V_1^{b_1}V_2^{b_2}$ in $\mathcal{G}(I_{\W}^{(m)}),$ we have to show that $F\in I_\W^m.$ We will show that we can decompose the form $F$ as $F_1\cdot F_2$ with $F_1\in I_\W$ and $F_2\in I_\W^{(m-1)},$ therefore the theorem will follows by induction. Let us consider the Euclidean division in $\mathbb{N}$ of $a_i, b_j$ with $m,$ say  $a_i=c_im+r_i,$ 
    $b_j=d_jm+s_j,$ for $i,j\in \{1,2\}$. 
    We get, for $(i,j)\in \{(1,1), (1,2), (2,1) \}$
    $c_im+r_i+d_jm+s_j=a_i+b_j\ge m\cdot m_{ij},$ i.e. $c_i+d_j+(r_i+s_j)/ m \ge  m_{ij}$ then $$c_i+d_j+\left\lfloor (r_i+s_j)/ m\right\rfloor \ge  m_{ij}.$$

    Let $\delta_i= \begin{cases}
    1& \text{if\ } s_i \neq 0\\
    0 & \text{if\ } s_i =0
    \end{cases}$ and set $$F_1:=H_1^{c_1}H_2^{c_2}V_1^{d_1+\delta_1}V_2^{d_2+\delta_2}\ \ \text{ and}\ \  F_2:=H_1^{a_1-c_1}H_2^{a_2-c_2}V_1^{b_1-d_1-\delta_1}V_2^{b_2-d_2-\delta_2}.$$

    For $(i,j)\in \{(1,1), (1,2), (2,1) \}$ we have $c_i+d_j +\delta_j\ge  c_i+d_j+\left\lfloor (r_i+s_j)/ m\right\rfloor \ge  m_{ij}.$
    This guarantees that $F_1\in I_\W.$ Analogously $a_i-c_i+b_j-d_j-\delta_j= c_i(m-1)+r_i + d_j(m-1)+(s_j-1)_+=$
    $(m-1)(c_1 + d_1+ (r_i+(s_j-1)_+)/(m-1) ).$
    Since   $(r_i+(s_j-1)_+)/(m-1)\ge \left\lfloor (r_i+s_j)/ m\right\rfloor $ we are done.
\end{proof}

We are ready to prove the main result of this section. Set
$I_{\mathcal{Z}}^{(m)}:=(Q_1,U_1)^{m_{11}m}\cap(Q_1,U_2)^{m_{12}m}\cap(Q_2,U_1)^{m_{21}m},$
i.e. ${\mathcal{Z}}^{(m)}$ is 
\begin{center}
    \begin{picture}(150,110)(25,-10)
    \put(5,30){$\mathcal{Z}^{(m)} = $}
    \put(60,-5){\line(0,1){80}}
    \put(100,-5){\line(0,1){80}}
    \put(140,35){\line(0,1){40}}


    \put(60,-5){\line(1,0){40}}
    \put(60,35){\line(1,0){80}}
    \put(60,75){\line(1,0){80}}


    \put(68,13){$m_{21}m$} \put(63,18){$$}

    \put(68,53){$m_{11}m$} \put(83,58){$$}


    \put(108,53){$m_{12}m$} 


    \end{picture}
\end{center}
\begin{theorem}\label{compowerZZ}
Let $\mathcal Z$ be a fat ACI of type (\ref{fatZ}). Then $I_{\mathcal{Z}}^{m}= I_{\mathcal{Z}}^{(m)}.$
\end{theorem}
\begin{proof}Since Lemma \ref{ZZasZ} and Remark \ref{seeZ}, we can repeat the same argument as in Theorem \ref{compower}.
\end{proof}


\end{document}